\documentclass[11 pt,leqno]{amsart}
\theoremstyle{definition}
\usepackage{indentfirst, amssymb,amsmath,amsthm}
\usepackage{csquotes}
\usepackage{hyperref}
\newtheorem{theo}{Theorem}[section]

\newcommand{\ol}{\overline}
\newcommand{\be}{\begin{equation}}
\newcommand{\ee}{\end{equation}}
\newcommand{\beas}{\begin{eqnarray*}}
\newcommand{\eeas}{\end{eqnarray*}}
\newcommand{\bea}{\begin{eqnarray}}
\newcommand{\eea}{\end{eqnarray}}

\numberwithin{equation}{section}
\begin{document}
\title[Unique range sets without Fujimoto's hypothesis]{Unique range sets without Fujimoto's hypothesis}
\date{}
\author[B. Chakraborty]{Bikash Chakraborty}
\date{}
\address{$^{1}$Department of Mathematics, Ramakrishna Mission Vivekananda Centenary College, Rahara,
West Bengal 700 118, India.}
\email{bikashchakraborty.math@yahoo.com, bikash@rkmvccrahara.org}
\maketitle
\let\thefootnote\relax
\begin{abstract} This paper studies the uniqueness of two nonconstant meromorphic functions when they share a finite set. Moreover, we will give the existence of unique range sets for meromorphic functions that are zero sets of some polynomials that do not necessarily satisfy the Fujimoto's hypothesis (\cite{Fu}).
\end{abstract}
\section{Introduction}
We use $M(\mathbb{C})$  to denote the set of all meromorphic  functions in $\mathbb{C}$. Let $S\subset\mathbb{C}\cup\{\infty\}$ be a non-empty set with distinct elements. Further suppose that $f,~g$ be two non-constant meromorphic (resp. entire) functions. We set $$E_{f}(S)=\bigcup\limits_{a\in S}\{z~:~f(z)-a=0\},$$
where a zero of $f-a$ with multiplicity $m$ counts $m$ times in $E_{f}(S)$. If $E_{f}(S)=E_{g}(S)$, then we say that $f$ and $g$ share the set $S$ CM.\par If $E_{f}(S)=E_{g}(S)$ implies $f\equiv g$, then the set $S$ is called a \emph{unique range set} for meromorphic (resp. entire) functions, in short, URSM (resp. URSE).\par
The first example of a unique range set was given by  F. Gross and C. C. Yang (\cite{GY}). They proved that if two non-constant entire functions $f$ and $g$ share the set $S=\{z\in \mathbb{C}:e^{z}+z=0\}$ CM, then $f\equiv g$. Since then, many efforts
were made to construct new unique range sets with cardinalities as small as possible (see chapter $10$ of \cite{YY}).\par
So far, the smallest  URSM has $11$ elements which was constructed by G. Frank and M. Reinders (\cite{FR}). That URSM is the zero set of the following polynomial.
\begin{equation}\label{fr}
P(z)=\frac{(n-1)(n-2)}{2}z^{n}-n(n-2)z^{n-1}+\frac{n(n-1)}{2}z^{n-2}-c,
\end{equation}
where $n\geq 11$ and $c(\not= 0,1)$ is any complex number.\par
To characterize the unique range sets, in 2000, H. Fujimoto (\cite{Fu}) made a major breakthrough by observing that almost all \emph{unique range sets  are the zero sets of some polynomials which satisfy an injectivity condition} (which is known as Fujimoto's hypothesis). To state his result, we recall some well-known definitions.\par
Let $P(z)$ be a non-constant monic polynomial in $\mathbb{C}[z]$. The polynomial $P(z)$ is called a \emph{uniqueness polynomial} for meromorphic (resp. entire) functions, in short, UPM (resp. UPE)  if the condition $P(f)\equiv P (g)$ implies $f \equiv g$ where $f$ and $g$ are any two non-constant meromorphic (resp. entire) functions.\par
Also, the polynomial $P(z)$ is called a \emph{strong uniqueness polynomial} for meromorphic (resp. entire) functions, in short, SUPM (resp. SUPE) if the condition  $P(f)\equiv cP (g)$ implies $f \equiv g$ where $f$ and $g$ are any two non-constant meromorphic (resp. entire) functions  and  $c$ is any non-zero complex number.\par
Thus strong uniqueness polynomials are uniqueness polynomials but the converse is not true, in general. For example,  we consider the polynomial $P(z)=az+b$ $(a\neq 0)$. Then for any non-constant meromorphic function (resp. entire) $g$ if we take $f := cg-\frac{b}{a}(1-c)$  $(c \neq 0,1)$, then we see that $P(f)=cP(g)$ but $f \neq g$.\par
Let $P(z)$ be a polynomial such that its derivative $P'(z)$ has $k$ distinct zeros $d_1, d_2, \ldots, d_k$ with multiplicities $q_1, q
_2,\ldots,q_{k}$ respectively. The polynomial $P(z)$ is said to satisfy \enquote{condition H} (\cite{Fu}) (which is known as Fujimoto's hypothesis) if
\begin{equation}\label{jkms2}
  P(d_{l_{s}})\not=P(d_{l_{t}})~~~~(1\leq l_{s}< l_{t}\leq k),
\end{equation}
Now, we state Fujimoto's (\cite{Fu}) result.
\begin{theo}\label{1111}(\cite{Fu})
 Let $P(z)$ be a strong uniqueness polynomial of the form $P(z)=(z-a_{1})(z-a_{2})\ldots(z-a_{n})$ ($a_{i}\not=a_{j}$) satisfying the condition (\ref{jkms2}). Moreover, either $k\geq3$ or $k=2$ and $\min\{q_{1},q_{2}\}\geq 2$. If $S=\{a_{1}, a_{2},\ldots, a_{n}\}$, then $S$ is a URSM (resp. URSE) whenever $n\geq2k+7$ (resp. $n\geq2k+3$).
\end{theo}
But, in 2011, T. T. H. An (\cite{An}) constructed a URSM that is the zero set of a polynomial which is not necessarily satisfying the Fujimoto's hypothesis (\ref{jkms2}).
\begin{theo}\label{B1}(\cite{An}) Let $P(z)=a_{n}z^{n}+a_{m}z^{m}+a_{m-1}z^{m-1}+\ldots+a_{0}$, $(1\leq m<n,~a_{i}\in \mathbb{C}, ~\text{and}~a_{m}\not=0)$ be a polynomial of degree $n$ with only simple zeros, and let $S$ be its zero set. Further suppose that $k$ be the number of distinct zeros of the derivative $P'(z)$ and $I=\{i~:~a_{i}\not=0\}$, $\lambda=\min\{i~:~i\in I\}$, $J=\{i-\lambda~:~i\in I\}$. If $n\geq \max\{2k+7,m+4\}$, then the following statements are equivalent:
\begin{enumerate}
\item [i)] $S$ is a URSM.
\item [ii)] $P$ is a SUPM.
\item [iii)] $S$  is affine rigid.
\item [iv)] The greatest common divisors of the indices respectively in $I$ and $J$ are both $1$.
\end{enumerate}
\end{theo}
Later, in 2012, using the concept of \emph{weighted sharing} (\cite{L}), A. Banerjee and I. Lahiri constructed a unique range set that is the zero set of a polynomial which is not necessarily satisfying the Fujimoto's hypothesis. To state the result of Banerjee and Lahiri, we need to recall the definition of \emph{weighted set sharing} (\cite{BL}).\par
Let $f$ and $g$ be two nonconstant meromorphic functions and  $l$ be any non-negative integer or infinity. For $a\in\mathbb{C}\cup\{\infty\}$, we denote by $E_{l}(a;f)$, the set of all $a$-points of $f$, where an $a$-point of multiplicity $m$ is counted $m$ times if $m\leq l$ and $l+1$ times if $m>l$. If $E_{l}(a;f)=E_{l}(a;g)$, then we say that \emph{$f$ and $g$ share the value $a$ with weight $l$}.\par
For $S\subset \mathbb{C}\cup\{\infty\}$, we define $E_{f}(S,l)=\cup_{a\in S}E_{l}(a;f).$  If $E_{f}(S,l)=E_{g}(S,l)$, then we say that $f$ and $g$ share the set $S$ with weight $l$, or simply \emph{$f$ and $g$ share $(S,l)$}.\par If $E_{f}(S,l)=E_{g}(S,l)$ implies $f\equiv g$, then the set $S$ is called a \emph{unique range set for meromorphic (resp. entire) functions with weight $l$}, in short, URSM$_{l}$ (resp. URSE$_{l}$).
\begin{theo}\label{B}(\cite{BL}) Let $P(z)=a_{n}z^{n}+\sum\limits_{j=2}^{m}a_{j}z^{j}+a_{0}$ be a polynomial of degree $n$, where $n-m\geq 3$ and $a_{p}a_{m}\not=0$ for some positive integer $p$ with $2\leq p\leq m$ and $\gcd(p,3)=1$. Suppose further that $S=\{\alpha_{1},\alpha_{2},\ldots,\alpha_{n}\}$ be the set of all distinct zeros of $P(z)$. Let $k$ be the number of distinct zeros of the derivative $P'(z)$. If $n\geq 2k+7~(\text{resp.}~2k+3)$, then the following statements are equivalent:
\begin{enumerate}
\item [i)] $P$ is a SUPM (resp. SUPE).
\item [ii)] $S$ is a URSM$_{2}$ (resp. URSE$_{2}$).
\item [iii)] $S$ is a URSM (resp. URSE).
\item [iv)] $P$ is a UPM (resp. UPE).
\end{enumerate}
\end{theo}
We have seen from theorem \ref{B1} and theorem \ref{B} that the unique range set generating polynomial is a specific polynomial, i.e., the unique range set generating polynomial has a gap after $n$-th degree term (where $n$ is the degree of the respective polynomial). The motivation of this short note is to construct a family of new unique range sets such that the corresponding generating polynomials  are \emph{not necessarily satisfying the Fujimoto's hypothesis} as well as \emph{the generating polynomials have no \enquote{such} gap}.
\section{Main Results}
Let
\bea\label{eqn1.1} P(z)=z^{n}+a_{n-1}z^{n-1}+\ldots+a_{1}z+a_{0}, \eea
be a monic polynomial of degree $n$ in $\mathbb{C}[z]$ without multiple zeros. Let $P(z)-P(0)$ has $m_{1}$ simple zeros and $m_{2}$ multiple zeros. Further suppose that $P'(z)$ has $k$ distinct zeros.
\begin{theo}\label{thb1.1} Let $P(z)$ be a monic polynomial defined by (\ref{eqn1.1}) with $P(0)\not=0$. Suppose further that $S=\{\alpha_{1},\alpha_{2},\ldots,\alpha_{n}\}$ be the set of all distinct zeros of $P(z)$. If $k\geq 2$, $m_{1}+m_{2}\geq 5$ (resp. $3$) and $n\geq \max\{2k+7, m_{1}+m_{2}+3\}$ (resp. $n\geq \max\{2k+3, m_{1}+m_{2}+1\}$, then the following statements are equivalent:
\begin{enumerate}
\item [i)] $P$ is a SUPM (resp. SUPE).
\item [ii)] $S$ is a URSM$_{2}$ (resp. URSE$_{2}$).
\item [iii)] $S$ is a URSM (resp. URSE).
\end{enumerate}
\end{theo}
\begin{theo}\label{thb1.2} Let $P(z)=z^{n}+a_{n-1}z^{n-1}+\ldots+a_{1}z+a_{0}$ be a monic polynomial of degree $n$ in $\mathbb{C}[z]$ with $P(0)\not=0$. If $P(z)-P(0)$ has $m_{1}$ simple zeros  and $m_{2}$  multiple zeros, and $n\geq 2(m_{1}+m_{2})+2$ (resp. $n\geq 2(m_{1}+m_{2})+1$), then the following two statements are equivalent:
\begin{enumerate}
\item [i)] $P$ is a SUPM (resp. SUPE).
\item [ii)] $P$ is a UPM (resp. UPE).
\end{enumerate}
\end{theo}
\begin{proof}[\textbf{Proof of the theorem \ref{thb1.1}}]
Since, the two cases $(ii)\Rightarrow (iii)$ and $(iii)\Rightarrow (i)$ are straightforward, so we only prove that $(i) \Rightarrow (ii)$.\par
Assume that $P(z)$ is a SUPM (resp. SUPE) and $E_{f}(S,2)=E_{g}(S,2)$. Now, we put
$$F(z):=\frac{1}{P(f(z))}~~\text{and}~~G(z):=\frac{1}{P(g(z))}.$$
Let $S(r):(0,\infty)\rightarrow\mathbb{R}$ be any function satisfying $S(r)=o(T(r,F)+T(r,G))$ for $r\rightarrow\infty$ outside a set of finite Lebesgue measure. Next we suppose that
$$H(z):=\frac{F''(z)}{F'(z)}-\frac{G''(z)}{G'(z)}.$$
First we assume that $H\not\equiv 0$. The lemma of logarithmic derivative gives
\bea\label{sus}m(r,H)=S(r).\eea
By construction of $H$, $H$ has at most simple poles and poles of $H$ can only occur at poles of $F$ and $G$, and zeros of $F'$ or $G'$ (\cite{ckps}). Since $F$ and $G$ share $\infty$ with weight $2$, thus
\bea \label{equn1.3} N(r,\infty;H)&\leq& \sum_{j=1}^{k}\left(\overline{N}(r,\lambda_{j};f)+\overline{N}(r,\lambda_{j};g)\right)+\ol{N}_{0}(r,0;f')+\ol{N}_{0}(r,0;g') \\
\nonumber &&+ \ol{N}(r,\infty;f)+\ol{N}(r,\infty;g)+\ol{N}_{\ast}(r,\infty;F,G),\eea
where $\lambda_{1},\lambda_{2},\ldots,\lambda_{k}$ are the distinct zeros of $P'(z)$. ( Here we write $\ol{N}_{0}(r,0;f')$ for the reduced counting function of zeros of $f'$, which are not zeros of $\prod_{i=1}^{n}(f-\alpha_{i})\prod_{j=1}^{k}(f-\lambda_{j})$. Similarly $\ol{N}_{0}(r,0;g')$ is defined. Also we write $\ol{N}_{\ast}(r,\infty;F,G)$ to denote the  reduced counting function of those poles of $F$ whose multiplicities differ from the multiplicities of the corresponding poles of $G$.)\par
Now the Laurent series expansion of $H$ shows that $H$ has a zero at every  simple pole of $F$ (hence, that of $G$). Thus using the first fundamental theorem, we conclude that
\bea \label{equn1.1} N(r,\infty;F|=1)=N(r,\infty;G|=1)\leq N(r,0;H)\leq N(r,\infty;H)+S(r), \eea
where $N(r,\infty;F|=1)$ is the the counting function of simple poles of $F$. Thus combining the inequalities (\ref{equn1.3}) and (\ref{equn1.1}), we obtain
\bea\label{aaess}&&\ol{N}(r,\infty;F)+\ol{N}(r,\infty;G)-\ol{N}_{0}(r,0;f')-\ol{N}_{0}(r,0;g')\\
\nonumber &\leq& \sum_{j=1}^{k}\left(\overline{N}(r,\lambda_{j};f)+\overline{N}(r,\lambda_{j};g)\right)+\ol{N}(r,\infty;f)+\ol{N}(r,\infty;g)\\
\nonumber&&+\ol{N}(r,\infty;F|\geq 2)+\ol{N}(r,\infty;G)+\ol{N}_{\ast}(r,\infty;F,G)+S(r)\\
\nonumber &\leq& \sum_{j=1}^{k}\left(\overline{N}(r,\lambda_{j};f)+\overline{N}(r,\lambda_{j};g)\right)+\ol{N}(r,\infty;f)+\ol{N}(r,\infty;g)\\
\nonumber&&+\frac{1}{2}\{N(r,\infty;F)+N(r,\infty;G)\}+S(r).
\eea
The second fundamental theorem applied to $f$ and $g$ gives
\bea \label{equn1.444} && (n+k-1)\left(T(r,f)+T(r,g)\right)\\
\nonumber &\leq& \ol{N}(r,\infty;f)+\sum_{i=1}^{n}\ol{N}(r,\alpha_{i};f)+\sum_{j=1}^{k}\overline{N}(r,\lambda_{j};f)-\ol{N}_{0}(r,0;f')+\ol{N}(r,\infty;g)\\
\nonumber &&+\sum_{i=1}^{n}\ol{N}(r,\alpha_{i};g)+\sum_{j=1}^{k}\overline{N}(r,\lambda_{j};g)-\ol{N}_{0}(r,0;g')+S(r,f)+S(r,g)\\
\nonumber &\leq& \ol{N}(r,\infty;f)+\ol{N}(r,\infty;g)+\sum_{j=1}^{k}\left(\overline{N}(r,\lambda_{j};f)+\overline{N}(r,\lambda_{j};g)\right)\\
\nonumber &&+\ol{N}(r,\infty;F)+\ol{N}(r,\infty;G)-\ol{N}_{0}(r,0;f')-\ol{N}_{0}(r,0;g')+S(r)\\
\nonumber &\leq& 2\left(\ol{N}(r,\infty;f)+\ol{N}(r,\infty;g)\right)+(2k+\frac{n}{2})\left(T(r,f)+T(r,g)\right)+S(r),\eea
which contradicts the assumption $n\geq2k+7$ (resp. $n\geq2k+3$). Thus from now we assume that $H\equiv 0$.  Then by integration, we obtain
\bea\label{es1} \frac{1}{P(f(z))}&\equiv&\frac{c_0}{P(g(z))}+c_{1},\eea
where $c_{0}$ is a non zero complex constant. Thus
$$T(r,f)=T(r,g)+O(1).$$
Now we distinguish two cases:\\
\textbf{Case-I} Assume that $c_{1}\not=0.$ Then equation (\ref{es1}) can be written as
$$P(f)\equiv \frac{P(g)}{c_{1}P(g)+c_0}.$$
Thus $$\overline{N}(r,-\frac{c_0}{c_1};P(g))=\overline{N}(r,\infty;P(f))=\overline{N}(r,\infty;f).$$
Since $P(z)-P(0)$ has $m_{1}$ simple zeros and $m_{2}$ multiple zeros, so we can assume
$$P(z)-P(0)=(z-b_{1})(z-b_{2})\ldots(z-b_{m_{1}})(z-c_{1})^{l_{1}}(z-c_{2})^{l_{2}}\ldots(z-c_{m_{2}})^{l_{m_{2}}},$$
where $l_{i}\geq 2$ for $1\leq i\leq m_{2}$. Moreover, $l_{i}< n$ as $P'(z)$ has at least two zeros.  If $P(0)\not=-\frac{c_0}{c_1}$, then the first and second fundamental theorems to $P(g)$ give
\beas &&n T(r,g)+O(1)\\
&=&T\left(r, P(g)\right)\\
&\leq& \overline{N}\left(r,\infty;P(g)\right)+\overline{N}\left(r,P(0);P(g)\right)+\overline{N}\left(r,-\frac{c_0}{c_1};P(g)\right)+S(r,P(g))\\
&\leq& \overline{N}\left(r,\infty; g\right)+\overline{N}\left(r,\infty; f\right)+(m_{1}+m_{2})T(r,g)+S(r,g),\eeas
which is impossible as $n\geq m_{1}+m_{2}+3$ (resp. $n\geq m_{1}+m_{2}+1$). Thus $P(0)=-\frac{c_0}{c_1}$. Hence
$$P(f)\equiv \frac{P(g)}{c_{1}(P(g)-P(0))}.$$
Thus every zero of $g-b_{j}$ ($1\leq j \leq m_{1}$) has a multiplicity at least $n$, and every zero of $g-c_{i}$ ($1\leq i \leq m_{2}$) has a multiplicity at least $2$.\par Thus applying the second fundamental theorem to $g$, we have
\beas
&&(m_{1}+m_{2}-1)T(r,g)\\
&\leq& \overline{N}(r,\infty;g)+\sum\limits_{j=1}^{m_1}\overline{N}(r,b_{j};g)+\sum\limits_{i=1}^{m_2}\overline{N}(r,c_{i};g)+S(r,g)\\
&\leq& \overline{N}(r,\infty;g)+\frac{1}{n}\sum\limits_{j=1}^{m_1}N(r,b_{j};g)+\frac{1}{2}\sum\limits_{i=1}^{m_2}N(r,c_{i};g)+S(r,g)\\
&\leq& \overline{N}(r,\infty;g)+\frac{m_1+m_{2}}{2}T(r,g)+S(r,g),
\eeas
which is impossible as $m_{1}+m_{2}\geq 5$ (resp. 3).\\
\textbf{Case-II} Next we assume that $c_{1}=0$. Then equation (\ref{es1}) can be written as
$$P(g)\equiv c_{0} P(f).$$ Since $P$ is a strong uniqueness polynomial, thus
$$f\equiv g.$$
This completes the proof.
\end{proof}
\begin{proof}[\textbf{Proof of the theorem \ref{thb1.2}}]
Since strong uniqueness polynomials are uniqueness polynomials, so we only prove the case $(ii)\Rightarrow (i).$ It is given that $P(z)$ is a uniqueness polynomial. Assume that $$P(g)=c_{0}P(f),$$
where $f$ and $g$ are two non-constant meromorphic functions and $c_{0}$ is any non-zero complex constant. Thus $T(r,f)=T(r,g)+O(1)$. Now, if $c_{0}\not=1$, then
$$P(g)-P(0)\equiv c_{0}(P(f)-\frac{P(0)}{c_0}).$$
Thus using the first and second fundamental theorems to $P(f)$, we obtain
\beas &&n T(r,f)+O(1)\\
&=&T\left(r, P(f)\right)\\
&\leq& \overline{N}\left(r,\infty;P(f)\right)+\overline{N}\left(r,P(0);P(f)\right)+\overline{N}\left(r,\frac{P(0)}{c_{0}};P(f)\right)+S(r,f)\\
&\leq& \overline{N}\left(r,\infty;f\right)+2(m_{1}+m_{2})T(r,f)+S(r,f),
\eeas
which contradicts to our assumptions on $n$. Thus $c_0=1$, i.e.,
$$P(f)\equiv P(g).$$
Since $P(z)$ is a uniqueness polynomial, so $f\equiv g$.  This completes the proof.
\end{proof}
\begin{center} {\bf Acknowledgement} \end{center}
The author is grateful to the anonymous referees for their valuable suggestions which considerably improved the presentation of the paper.\par
The research work is supported by the Department of Higher Education, Science and Technology \text{\&} Biotechnology, Govt. of West Bengal under the sanction order no. 216(sanc) /ST/P/S\text{\&}T/16G-14/2018 dated 19/02/2019.

\end{document}